\documentclass[12pt]{amsart}

\author{Stefano Decio, Eugenia Malinnikova}
\address{S.D.: School of Mathematics, University of Minnesota, Minneapolis, MN, USA}
\email{decio001@umn.edu}
\address{E.M.: Department of Mathematics, Stanford University, Stanford, CA, USA; Department of Mathematical Sciences, Norwegian University of Science and Technology, Trondheim, Norway}
\email{eugeniam@stanford.edu}

\usepackage{amsmath}
\usepackage{amssymb}
\usepackage{amsfonts}
\usepackage{enumerate}
\usepackage{mathrsfs}
\usepackage{MnSymbol}
\usepackage{mathtools}
\usepackage{graphicx,wrapfig,lipsum}
\usepackage{hyperref}

\newcommand{\R}{{\mathbf R}}

\newcommand{\ld}{{\lambda}}
\newcommand{\defeq}{\vcentcolon=}

\newtheorem{theorem}{Theorem}
\newtheorem{lemma}{Lemma}
\newtheorem{claim}{Claim}

\newtheorem{proposition}{Proposition}
\newtheorem{thmx}{Theorem}

\theoremstyle{remark}
\newtheorem*{remark}{Remark}
\theoremstyle{definition}

\DeclareMathOperator{\diver}{div}

\begin{document}

\begin{abstract}
    Let $\varphi_{\ld}$ be an eigenfunction of the Laplace-Beltrami operator on a smooth compact Riemannian manifold $(M,g)$, i.e.,  $\Delta_g \varphi_{\ld} + \ld \varphi_{\ld}=0$. We show that $\varphi_{\ld}$ satisfies a local Bernstein inequality; namely for any geodesic ball $B_g(x,r)$ in $M$ there holds: $\sup_{B_g(x,r)}|\nabla\varphi_{\ld}|\leq C_{\delta}\max\left\{\frac{\sqrt{\ld}\log^{2+\delta}\ld}{r},\ld\log^{2+\delta}\ld\right\}\sup_{B_g(x,r)}|\varphi_{\ld}|$. We also prove analogous inequalities for solutions of elliptic PDEs in terms of the frequency function. 
\end{abstract}

\title{On a Bernstein inequality for eigenfunctions}
\maketitle

\section{Introduction}
An eigenfunction of the Laplace-Beltrami operator on a compact smooth Riemannian manifold $(M,g)$ is a solution of the equation 
\begin{align}
\label{eigen}
    \Delta_g\varphi_{\ld}+\ld\varphi_{\ld}=0. 
\end{align}
A pervasive theme in the study of the local properties of eigenfunctions is that $\varphi_{\ld}$ should behave like a polynomial of degree approximately $\sqrt{\ld}$.  This idea informs the influential conjecture of Yau on the size of the nodal set of eigenfunctions and the works of Donnelly and Fefferman (\cite{DF},\cite{DF2}). In this note we pursue the analogy and obtain fine estimates on the growth of eigenfunctions. In particular, we seek to obtain a version of the classical Bernstein inequality; for a trigonometric polynomial $T_N$ of degree $N$, the inequality reads
\begin{align*}
    \sup_{\theta\in (-\pi,\pi)}|T'_N(\theta)|\leq N \sup_{\theta\in (-\pi,\pi)}|T_N(\theta)|.
\end{align*}
For an algebraic polynomial in $P_N$, also of degree $N$, the following inequality, usually called the Markov or Bernstein-Markov inequality, holds:
\begin{align*}
    \sup_{x\in (-1,1)}|P'_N(x)|\leq N^2 \sup_{x\in (-1,1)}|P_N(x)|.
\end{align*}
Note that in general one needs $N^2$ in the Markov inequality; the Chebyshev polynomials are extremizers of the inequality. Essentially, the Bernstein inequality for trigonometric polynomials is a global inequality, while the Markov inequality is local in nature. For eigenfunctions, the global Bernstein inequality
\begin{align}
    \label{global}
    \sup_{M}|\nabla \varphi_{\ld}|\leq C_M\sqrt{\ld} \sup_{M}|\varphi_{\ld}|
\end{align}
holds and is an easy consequence of standard elliptic estimates for PDEs; see for instance \cite[Corollary 3.3]{OP} for a proof. Let us remark that a global Bernstein inequality also holds for linear combinations of eigenfunctions (\cite{FM}), but the proof is considerably more involved. 

A local version of the Bernstein inequality is substantially more difficult to prove; the study was initiated by Donnelly and Fefferman in \cite{DF2}, where they prove that
\begin{equation}
\label{eq:DFL2}
\int_{B(x,(1+1/\sqrt{\ld})r)}|\varphi_\ld|^2\lesssim \int_{B(x,r)}|\varphi_\ld|^2,
\end{equation}
and deduce the following result
\begin{thmx}
\label{df}
Let $(M,g)$ be a smooth compact Riemannian manifold of dimension $d$ and let $\varphi_{\ld}$ be a solution of \eqref{eigen}. Then 
\begin{align*}
    \sup_{B(x,r)}|\nabla\varphi_{\ld}|\leq C(M,g)\frac{\ld^{\frac{d+2}{2}}}{r}\sup_{B(x,r)}|\varphi_{\ld}|
\end{align*}
for any geodesic ball $B(x,r)$.
\end{thmx}
They also conjecture that the inequality should hold with exponent $1/2$ instead of $(d+2)/2$, i.e. that 
\begin{align*}
    \sup_{B(x,r)}|\nabla\varphi_{\ld}|\leq C(M,g)\frac{\sqrt{\ld}}{r}\sup_{B(x,r)}|\varphi_{\ld}|
\end{align*}
We note that this, if true, would give an inequality that is stronger than the Markov inequality for polynomials, where the square of the degree of the polynomial appears; imprecisely but suggestively, the conjecturally sharp inequality would tell us that an eigenfunction $\varphi_{\ld}$ looks like a \emph{harmonic} polynomial of degree $\sqrt{\ld}$. In \cite {Do} Dong obtained a sharper version of Theorem \ref{df} for eigenfunctions on surfaces, using powerful ideas introduced in \cite{Do2}, see Section \ref{s2d} for the details.

The main result of the present article is an improvement of Theorem \ref{df} in any dimension, which, up to logarithmic errors, gives the conjecture of Donnelly and Fefferman at scales up to the wavelength $\frac{1}{\sqrt{\lambda}}$, and resembles more the Markov inequality at larger scales.

\begin{theorem}
\label{main}
Let $(M,g)$ be a compact smooth Riemannian manifold, and let $\varphi_{\ld}$ be a solution of $\Delta_g \varphi_{\ld} + \ld \varphi_{\ld}=0$. Let $B_g(x,r)\subset M$ be any geodesic ball. Let $\delta>0$ be arbitrary. Then the following Bernstein-type inequality holds:
\begin{align}
\label{bernstein}
    \sup_{B_g(x,r)}|\nabla\varphi_{\ld}|\leq C(M,g,\delta)\max\left\{\frac{\sqrt{\ld}\log^{2+\delta}\ld}{r},\ld\log^{2+\delta}\ld\right\}\sup_{B_g(x,r)}|\varphi_{\ld}|
\end{align}
\end{theorem}
\begin{remark}
    \textit{(i)} The constant $C(M,g,\delta)$ blows up as $\delta$ goes to $0$. If the reader wishes to forget about $\delta$, it is certainly possible to think that $\delta=1$ and discard the $\delta$-dependence of the constants. \\
    \textit{(ii)} In dimension two, using some geometric techniques of Dong, we can get the sharp inequality up to scale $(\log\ld)^{-1}$; see Theorem \ref{2dmain} in Section \ref{s2d} for the precise statement. \\
    \textit{(iii)} We also obtain an $L^{\infty}$ version of \eqref{eq:DFL2}, which is effectively equivalent to the Bernstein inequality.
\end{remark}

Our proof of Theorem \ref{main} hinges on a version of the Bernstein inequality for solutions of second order elliptic PDEs with smooth enough coefficients. The role of (square root of) the eigenvalue is played in this case by the frequency function (or the doubling index), which serves as a local degree of the solution. We believe the following result is of independent interest, since it shows that a solution to an elliptic PDE with bounded frequency also behaves like a polynomial. By $N_u(0,2)$ in the statement below we denote the frequency function of $u$ in the ball $B(0,2)$, see \eqref{frequency} at the beginning of Section \ref{sell} for the definition.
\begin{theorem}
\label{ellipticmain}
let $A(x)=(a_{ij}(x))$ be a uniformly positive definite matrix satisfying
\begin{enumerate}[(i)]
    \item  $\Lambda^{-1}|\xi|^2 \leq A(x)\xi\cdot\xi \leq \Lambda|\xi|^2$ for any $\xi\in \R^d$, with $0<\Lambda<\infty$;
    \item $A(0)=Id$.
\end{enumerate} 
Let $u\in W^{1,2}(B(0,2))$ be a solution of $\diver(A\nabla u)=0$ in $B(0,2)$, and assume $N_u(0,2)\leq N$. Suppose that $N>10$ and let $L\vcentcolon=N\log^{2+\delta} N$, with $\delta>0$ arbitrarily small. Let $0<\eta\leq 1$ be fixed. Then, if for $k$ satisfying
\begin{align}
    \label{regularity}
    k\eta>\frac{d}{2}+1
\end{align}
$a^{ij}(x)$ is a $\mathcal{C}^{k}$ function of $x$ with 
\begin{align*}
    \|a^{ij}(x)\|_{\mathcal{C}^k}\leq \gamma \text{ for all } i,j,
\end{align*} 
the following inequalities hold for any $r<a/N^{\eta}$, where $a$ is a small constant depending on $\Lambda$ and $\gamma$:
\begin{align}
    \label{aux1}
    \sup_{B(0,(1+1/L)r)}|u|\leq C \sup_{B(0,r)}|u|;\\
    \label{aux2}
    \sup_{B(0,r)}|\nabla u| \leq \frac{CL}{r}\sup_{B(0,r)}|u|.
\end{align}
The constant $C$ above depends on $\Lambda,\gamma$ and $\delta$ but not on $N$.
\end{theorem}

\begin{remark}
    For the application of Theorem \ref{ellipticmain} to eigenfunctions we will only need $\eta=1$. But it is worth noting that for the equations considered here we do not need to take such small radii. 
\end{remark}

\subsection*{Plan of the paper}
Section \ref{sell} is devoted to the proof of Theorem \ref{ellipticmain}, which uses an auxiliary result for solutions of equations with analytic coefficient. Theorem \ref{main} is deduced from Theorem \ref{ellipticmain} in Section \ref{seigen}. In Section \ref{s2d} we give a proof of the aforementioned Theorem \ref{2dmain}, containing a sharper version of the Bernstein inequality in dimension two. In the Appendix we prove sharp $L^p$ Bernstein inequalities for harmonic functions in $\R^d$.

\subsection*{Notation} The dependence of the constants on the dimension is pervasive throughout the paper and will more often than not be suppressed. We indicate by $c,C$ constants whose value may change from line to line (and even within the same line) but only depends on the manifold $(M,g)$ or on the parameters $\Lambda,\gamma$. Sometimes we do not indicate the constants explicitly and use the symbols $\gtrsim,\lesssim,\sim$ with the same meaning. 

\section{Elliptic PDEs}
\label{sell}
Donnelly and Fefferman in \cite{DF2} prove Theorem \ref{df} directly for eigenfunctions using Carleman inequalities. In this work we take a different route: we view an eigenfunction at scales below the wavelength as a solution to a uniformly elliptic PDE through a well known lifting trick (see Section \ref{seigen}), and to prove the elliptic result we rely on the monotonicity of the frequency function and some 1-dimensional complex analysis. Let us now describe in detail the proof.

We let $A(x)=(a_{ij}(x))$ be as in the statement of Theorem \ref{ellipticmain}, meaning that 
\begin{enumerate}[(i)]
    \item  $\Lambda^{-1}|\xi|^2 \leq A(x)\xi\cdot\xi \leq \Lambda|\xi|^2$ for any $\xi\in \R^d$.
    \item  for some $k$ such that $k\eta>\frac{d}{2}+1$, $\|a^{ij}(x)\|_{\mathcal{C}^k}\leq \gamma$ for all $i,j$.
\end{enumerate}
We remind the reader that $0<\eta\leq 1$ is the parameter for which we seek to obtain inequalities \eqref{aux1} and \eqref{aux2} for $r\lesssim N^{-\eta}$. Let now 
\begin{align*}
    \mu(x)=\frac{A(x)x\cdot x}{|x|^2}
\end{align*}
and note that $\Lambda^{-1}\leq \mu(x) \leq \Lambda$. Following \cite{GL1} (see also \cite{LM1}), for a solution $u$ of $\diver(A\nabla u)=0$ in $B(0,2)\subset \R^d$, $u\in W^{1,2}(B(0,2))$, with $A(0)=Id$ and a ball $B(x,r)\subset B(0,2)$, we define the frequency function as
\begin{align}
\label{frequency}
    N_u(x,r)=\frac{r\int_{B(x,r)}A\nabla u\cdot\nabla u}{2\int_{\partial B(x,r)}\mu|u|^2}.
\end{align}
Recall that the key assumption of Theorem \ref{ellipticmain} is that $N_u(0,2)\leq N$.

Throughout the section, we let $B\vcentcolon=B(0,1)$. For the proof of Theorem \ref{ellipticmain} we will need an auxiliary result about solutions to elliptic PDEs with real-analytic coefficients; we will approximate $u$ with such functions. We state and prove the result in the next subsection.
\subsection{A lemma on analytic solutions}
We let here $\widetilde{A}$ be an elliptic matrix with entries $\widetilde{a}^{ij}$ real-analytic in $B$, and ellipticity constant $\Lambda$. More precisely, we require the existence of numbers $K,A$ such that 
\begin{align}
  \label{derivatives}  \max_{i,j}|\nabla^p \widetilde{a_{ij}}(x)|\leq Kp!A^p \text{ for all } p\geq0.
\end{align}
\begin{lemma}
\label{analytic}
    Consider a solution $v$ of the equation $\diver(\widetilde{A}\nabla v)=0$ in $B$, with $\widetilde{A}(0)$ satisfying \eqref{derivatives} and $\widetilde{A}(0)=Id$, and suppose that there is a number $N>10$ such that 
    \begin{enumerate}[(a)]
        \item $N_v(B)\leq N$;
        \item $\sup_B|v|\leq N^{\kappa}\|v\|_{L^2(\partial B)}$, with $\kappa\geq0$ a dimensional constant.
    \end{enumerate}
    Let $L\vcentcolon=N\log^{2+\delta}N$, with $\delta>0$ a real number. Then there is a constant $c=c(\Lambda,K,A,\kappa,\delta)$ such that 
    \begin{align}
        \sup_{(1-1/L)(1-1/N)B}|v|\geq c\sup_{(1-1/N)B}|v|.
    \end{align}
    \begin{proof}
        It is a classical result that a solution to an elliptic equation with coefficients satisfying \eqref{derivatives} extends to a holomorphic function in a domain in $\mathbf{C}^d$. It follows from the proof of Theorem 5.7.1' in \cite{Mo} that for every $y\in B$, there is a complex ball in $\mathbf{C}^d$ of radius equal to $c(\Lambda,K,A)dist(y,\partial B)$ such that $v$ extends uniquely to a holomorphic function $F$ in that ball, with the supremum of $F$ in the complex ball bounded by a multiple of the supremum of $v$ in the real ball. Call $\Omega\subset\mathbf{C}^d$ the domain obtained by the union of all these balls, and call $F$ the function holomorphic in $\Omega$ that extends $v$. By condition \textit{(b)} we have that 
        \begin{align*}
            \sup_{\Omega}|F|\leq CN^{\kappa}\|v\|_{L^2(\partial B)}.
        \end{align*}
        Let $m=\sup_{(1-1/L)(1-1/N)B}|v|$, with $L$ as in the statement of the Lemma. A consequence of almost monotonicity of the frequency and the bound \textit{(a)} is that $\int_{B(0,r_2)}|v|^2\lesssim (r_2/r_1)^{C(1+N)}\int_{B(0,r_1)}|v|^2$ for $r_1<r_2\leq 1$ (see \cite{HL}, Theorem 3.1.3). Then, since $L>N$, 
        \begin{align*}
            m\geq \sup_{(1-5/N)B}|v|\gtrsim \left(\int_{\partial((1-5/N)B)}|v|^2\right)^{\frac{1}{2}}\gtrsim \|v\|_{L^2(\partial B)}.
        \end{align*}
        In view of this, we normalize $v$ so that $C_0m=1$, where $C_0$ is the constant such that $\sup_{(1-1/L)(1-1/N)\Omega}|F|\leq C_0m$. We abuse notation and still denote by $v,F$ the normalized functions. We have then $\|v\|_{L^2(\partial B)}\leq C$, $\sup_{\Omega}|F|\leq CN^{\kappa}$, $\sup_{(1-1/L)(1-1/N)\Omega}|F|\leq 1$ and we need to prove that 
        \begin{align}
        \label{middle}
            \sup_{(1-1/N)B}|v|\leq C. 
        \end{align}
        Let $x_0\in \overline{(1-1/N)B}$ be such that $|v(x_0)|=\sup_{(1-1/N)B}|v|$. We take the line joining the origin and $x_0$ and consider the plane given by the complexification of it. Let $\Omega_0$ be the intersection of $\Omega$ with this plane. Up to an isometry we can assume that $x_0\in \mathbf{R}$ and $\Omega_0\subset\mathbf{C}$. Let $x_1\vcentcolon=(1-1/L)(1-1/N)\in \mathbf{R}$. The domain $\Omega_0$ contains the domain $R$ bounded by the contour $\Gamma=\Gamma_1\cup\Gamma_2\cup\Gamma_3\cup\Gamma_4$ where
        \begin{align*}
            &\Gamma_1=\{z=x+iy\in\mathbf{C} : x\geq 0, y\geq 0, z-1=\rho e^{i(\pi-\alpha)}\};\\
            &\Gamma_2=\{z=x+iy\in\mathbf{C} : x\geq 0, y\geq 0, z+1=\rho e^{i\alpha}\};\\
            &\Gamma_3=\{z=x+iy\in\mathbf{C} : x\leq 0, y\leq 0, z+1=\rho e^{-i\alpha}\};\\
            &\Gamma_4=\{z=x+iy\in\mathbf{C} : x\geq 0, y\leq 0, z-1=\rho e^{i(\alpha-\pi)}\}.
        \end{align*}
        for some $0<\alpha<\pi/2$. Note that $\alpha$ depends on $\Lambda,K,A$ only. Let us define $f\vcentcolon=F_{|_{\Omega_0}}$; we have that $\sup_{\Omega_0}|f|\leq CN^{\kappa}$ and $\sup_{(1-1/L)(1-1/N)\Omega_0}|f|\leq 1$. Finally, we set $h(z)\vcentcolon=\log|f(z)|$. Note that $h$ is a subharmonic function. Consider now the domain 
        \[\widetilde{\Omega}\vcentcolon=R\setminus (1-1/L)(1-1/N)R\cap B(x_1,1-x_1)\] and note that $x_0\in\widetilde{\Omega}.$ The boundary of $\widetilde{\Omega}$ consists of $\Gamma_-\cup\Gamma_+\cup\Gamma_0$,
        where
        \begin{align*}
            &\Gamma_-=\{
            x_1+\rho e^{i(\pi-\alpha)}, 0\le \rho\le 1-x_1\}\cup\{
            x_1+\rho e^{i(\alpha-\pi)}, 0\le\rho\le 1-x_1\};\\
            &\Gamma_+=(\Gamma_1\cup\Gamma_4)\cap B(x_1,1-x_1);\\
            &\Gamma_0=\partial B(x_1,1-x_1)\cap \left(R\setminus (1-1/L)(1-1/N)R\right).
        \end{align*}
        Below is a picture of the regions described: $\widetilde{\Omega}$ is the region bounded by the red curve.\\
        
        \includegraphics[width=320pt]{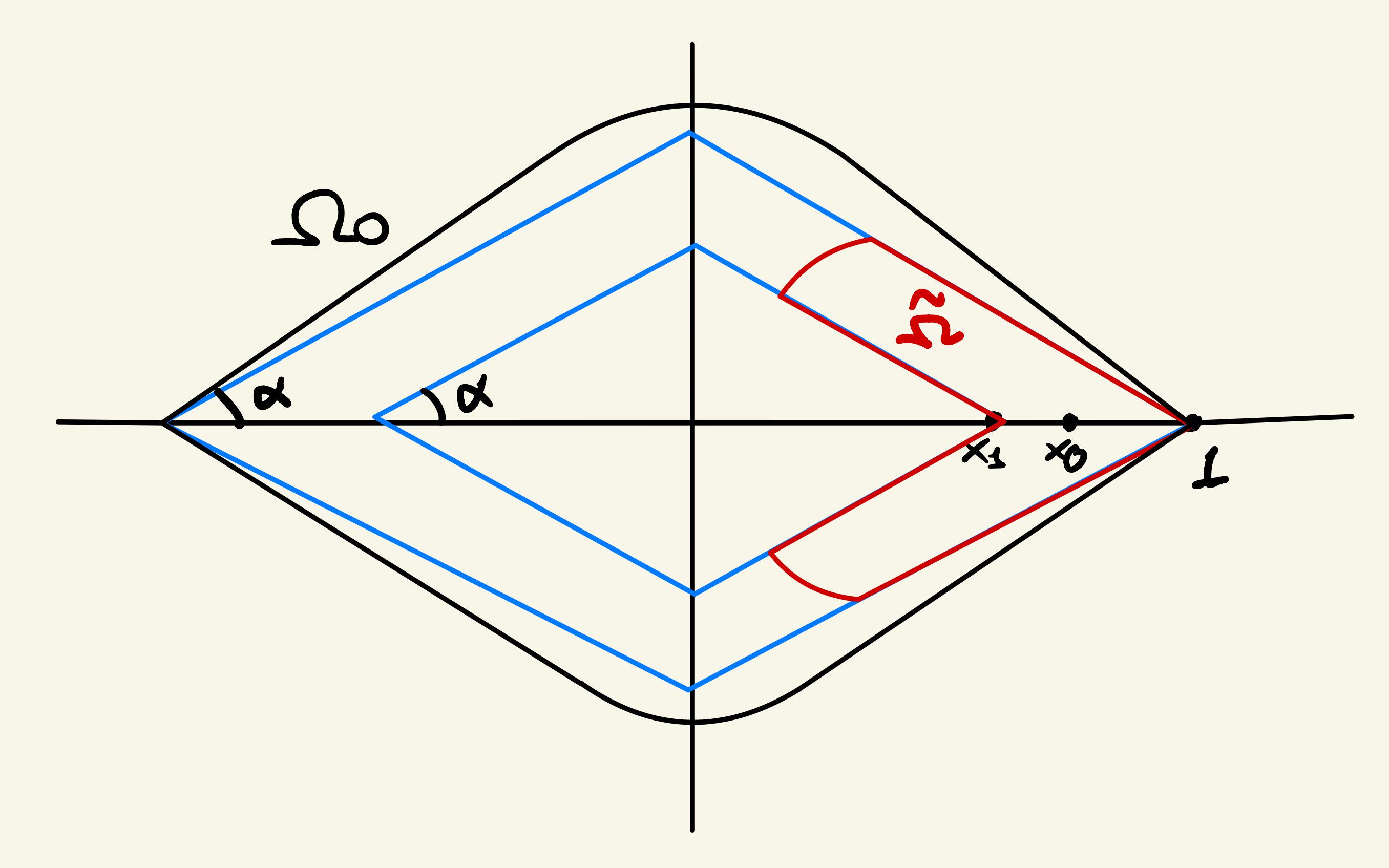} \\
        We introduce the harmonic function 
        \begin{align*}
            \omega(z)\vcentcolon=M\text{Re}\left(\left(\frac{z-x_1}{1-x_1}\right)^{\beta}\right),
        \end{align*}
        defined in $\mathbf{C}\setminus(-\infty, x_1)$, where $M>0$ and $\beta>0$ are real numbers to be chosen later. We want to use the maximum principle and bound $h$ by $\omega$; to this end, we estimate $\omega$ on $\partial\widetilde{\Omega}$. Note that on $\Gamma_-\cap\{y\geq 0\}$, $(z-x_1)^{\beta}=\rho^{\beta}e^{i\beta(\pi-\alpha)}$ and hence for any $\beta$ such that 
        \begin{align*}
            \beta(\pi-\alpha)<\frac{\pi}{2}
        \end{align*}
        we have $\text{Re}(z-x_1)^{\beta}\geq 0$ and $\omega(z)\geq 0$. We fix a choice for $\beta$: 
        \begin{align}
        \label{betafix}
            \beta=\frac{1}{2}-\widetilde{\delta}
        \end{align}
        with $\widetilde{\delta}$ a small positive real number to be chosen later. Note that since $\alpha>0$, $\beta$ chosen in this way is consistent with the condition above. This choice of $\beta$ guarantees that also on $\Gamma_-\cap\{y\leq 0\}$ we have $\omega(z)\geq 0$. Note that since $\sup_{(1-1/L)(1-1/N)\Omega_0}|f|\leq 1$, we have that $h\leq 0$ on $\Gamma_-$, so that $h(z)\leq\omega(z)$ for $z\in\Gamma_-$. 
        On $\Gamma_+$ we have 
        \[\min_{z\in\Gamma_+}\text{Re}\left(\left(\frac{z-x_1}{1-x_1}\right)^{\beta}\right)=\min_{0\le\theta\le\pi-2\alpha}\left(\frac{\sin \alpha}{\sin(\alpha+\theta)}\right)^\beta\cos\beta\theta\ge (\sin\alpha)^{\beta+1}\ge C_\alpha.\]
        Since $h(z)\leq \widetilde{C}+\kappa\log N$ on $\Gamma_+$, at the very least we require $MC_\alpha\geq \widetilde{C}+\kappa\log N$. 
        Finally, note that on $\Gamma_0$, $|z-x_1|=1-x_1$, and we write $z-x_1=(1-x_1)e^{i\theta}$. On $\Gamma_0\cap \{y\geq 0\}$, $\theta\in(\frac{\pi-\alpha}{2},\pi-\alpha)$ and by \eqref{betafix}
        \begin{align*}
            \text{Re}\left(\left(\frac{z-x_1}{1-x_1}\right)^{\beta}\right)=\cos(\beta\theta)\gtrsim \widetilde{\delta}.
        \end{align*}
        The same estimate holds on $\Gamma_0\cap \{y\leq 0\}$, hence $\omega(z)\geq M\widetilde{\delta}$ on $\Gamma_0$. 
        We choose $M$ such that 
        \begin{align}
        \label{M}
            M\widetilde{\delta}=C_1+C_2\kappa\log N
        \end{align}
        with $C_1$ and $C_2$ chosen in such a way that $h(z)\leq \omega(z)$ 
        on $\partial\widetilde{\Omega}$, so that $h(z)\leq\omega(z)$ on $\widetilde{\Omega}$ by the maximum principle. In particular,
        \begin{align*}
            h(x_0)\leq \omega(x_0)=M\text{Re}\left(\left(\frac{x_0-x_1}{1-x_1}\right)^{\beta}\right)=M\left(\frac{(1-1/N)1/L}{1-(1-1/N)(1-1/L)}\right)^{\beta}\\
            \leq M\left(\frac{N}{L}\right)^{\beta}=\left(\frac{C_1}{\widetilde{\delta}}+\frac{C_2\kappa\log N}{\widetilde{\delta}}\right)\left(\frac{1}{\log^{2+\delta}N}\right)^{\frac{1}{2}-\widetilde{\delta}}.
        \end{align*}
        Let us pick $\widetilde{\delta}=\min\{\delta/10,C_\alpha)$; then $h(x_0)\leq C_3/\delta$, and finally 
        \begin{align*}
            v(x_0)=f(x_0)=\exp(h(x_0))\leq \exp\left(\frac{C_3}{\delta}\right),
        \end{align*}
        which implies \eqref{middle} and concludes the proof.
    \end{proof}
\end{lemma}

\subsection{Proof of Theorem \ref{ellipticmain}}
Note first that \eqref{aux1} immediately implies \eqref{aux2} by elementary arguments (see for instance the proof of Lemma \ref{poly} in the Appendix, or the beginning of the proof of Proposition \ref{small} in Section \ref{seigen}), so we will prove \eqref{aux1} only. We start with 
\begin{claim}
\label{maxest}
    If $N_u(0,2)\leq N$, then 
    \begin{enumerate}[(i)]
    \item  $\max_{\overline{B}}|u|\lesssim N^{\frac{d}{2}}\left(\strokedint_{\partial B}|u|^2\right)^{\frac{1}{2}}$;
    \item $\max_{\overline{B}}|\nabla u|\lesssim N^{\frac{d}{2}+1}\left(\strokedint_{\partial B}|u|^2\right)^{\frac{1}{2}}$.
    \end{enumerate}
\begin{proof}
    We remind the reader that $B$ denotes the unit ball in $\R^d$. Let $x\in\overline{B}$ be such that $u(x)=\max_{\overline{B}}|u|$. By classical local boundedness properties of solutions to elliptic PDEs (see for instance \cite{GT}, Theorem 8.17),
    \begin{align*}
        |u(x)|^2\lesssim \strokedint_{B\left(x,\frac{1}{N}\right)}|u|^2 \lesssim N^d\int_{B\left(0,1+\frac{1}{N}\right)}|u|^2.
    \end{align*}
    We recall the inequality $\int_{B(0,r_2)}|u|^2\lesssim (r_2/r_1)^{C(1+N)}\int_{B(0,r_1)}|u|^2$ which was already used in the previous subsection. Applying it with $r_1=1$, $r_2=1+1/N$ in the above estimate, we obtain 
    \begin{align*}
        |u(x)|^2\lesssim N^d\left(1+\frac{1}{N}\right)^{CN}\int_{B}|u|^2,
    \end{align*}
    which gives \textit{(i)} in the claim. For \textit{(ii)}, the proof is the same but starting with the estimate $|\nabla u(x)|^2\lesssim N^2\strokedint_{B\left(x,\frac{1}{N}\right)}|u|^2$ instead.
\end{proof}
\end{claim}
Let now $r<a/N^{\eta}$ as in the statement of Theorem \ref{ellipticmain}, and define $f\vcentcolon=u_{|_{\partial B(0,r)}}$. Let $\widetilde{A}$ be the Taylor approximation of order $k$ of $A$ (this exists by our hypothesis on the regularity of $A$). We rescale everything to $B= B(0,1)$, denoting with a subscript $r$ the rescaled quantities. Note that $u_r(x)=u(rx)$ satisfies the equation $\diver(A_r\nabla u_r)=0$ in $B$, where $A_r$ satisfies the ellipticity estimates (i) with the same constants as $A$, and 
\begin{align}
\label{coeffest}
    \sup_{x\in B}\|A_r(x)-\widetilde{A}_r(x)\|\leq c(\gamma,d)r^k. 
\end{align}
Note that the entries of $\widetilde{A}_r(x)$ satisfy condition \eqref{derivatives} for some $K,A$ depending on $\gamma$. Let now $v$ be the solution of 
\begin{align}
    \begin{cases}
    \diver(\widetilde{A}_r\nabla v)=0 \qquad \: &\text{in} \ B,\\
    v=f_r &\text{on} \ \partial B. \end{cases}
\end{align}
We also set $h\vcentcolon=u_r-v$. We have that 
\begin{align*}
    \diver(\widetilde{A}_r\nabla h)=\diver((\widetilde{A}_r-A_r)\nabla u_r),
\end{align*}
and clearly $h_{|_{\partial B}}=0$. Let $G(x,y)$ denote Green's function for the operator $\diver(\widetilde{A}_r\nabla)$ in $B$; by the results in \cite{GW}, for any $x\in B$, $\nabla_y G(x,y)$ is absolutely integrable in $B$ with a bound on the $L^1$ norm depending on $\Lambda$ only. We can then estimate, for $x\in B$,
\begin{align*}
    |h(x)|=\left|\int_{B}\diver(\widetilde{A}_r\nabla h)(y)G(x,y)dy\right|=\left|\int_{B}\diver((\widetilde{A}_r-A_r)\nabla u_r)G(x,y)dy\right|\\
    =\left|\int_B (\widetilde{A}_r-A_r)\nabla u_r(y)\cdot \nabla_y G(x,y)dy\right|\leq C(\Lambda)\sup_B\|\widetilde{A}_r-A_r\|\sup_B|\nabla u_r|.
\end{align*}
We now apply \eqref{coeffest} and \textit{(ii)} in Claim \ref{maxest} on $u_r$ (noting that $N_{u_r}(B)=N_{u}(0,r)\leq CN$ by almost monotonicity of the frequency), obtaining
\begin{align*}
    |h(x)|\leq C(\Lambda,\gamma)r^k N^{\frac{d}{2}+1}\|f_r\|_{L^2(\partial B)}.
\end{align*}
Condition \eqref{regularity} tells us that $k\eta\geq d/2+1$, and remembering that $r<a/N^{\eta}$ we get 
\begin{align}
\label{aux}
    |h(x)|\leq \varepsilon \|f_r\|_{L^2(\partial B)},
\end{align}
with $\varepsilon$ a small constant, provided we choose $a$ small enough depending on $\Lambda, \gamma$.
\begin{claim}
\label{vestimates}
With $v$ as above, we have:
    \begin{enumerate}[(i)]
    \item $N_v(B)\leq CN$;
    \item  $\max_{\overline{B}}|v|\lesssim N^{\frac{d}{2}}\left(\int_{\partial B}|v|^2\right)^{\frac{1}{2}}$.
    \end{enumerate}
    \begin{proof}
        Note that we have $\int_B \widetilde{A}_r\nabla v\cdot\nabla h=0$, hence
        \begin{align*}
            \int_B \widetilde{A}_r\nabla v\cdot\nabla v=\int_B \widetilde{A}_r\nabla v\cdot\nabla u_r
        \end{align*}
        and we can estimate 
        \begin{align*}
            \int_B \widetilde{A}_r\nabla v\cdot\nabla v\lesssim \left(\int_B|\widetilde{A}_r\nabla v|^2\right)^{\frac{1}{2}}\left(\int_B|\nabla u_r|^2\right)^{\frac{1}{2}}\\
            \lesssim \left(\int_B\widetilde{A}_r\nabla v\cdot\nabla v\right)^{\frac{1}{2}}\left(\int_B A_r\nabla u_r\cdot\nabla u_r\right)^{\frac{1}{2}}
        \end{align*}
        so that 
        \begin{align*}
            \int_B \widetilde{A}_r\nabla v\cdot\nabla v\lesssim \int_B A_r\nabla u_r\cdot\nabla u_r,
        \end{align*}
        with implied constants depending on $\Lambda, \gamma$. Since on $\partial B$, $v=u_r=f_r$, for the frequency we have 
        \begin{align*}
            N_v(B)=\frac{\int_B\widetilde{A}_r\nabla v\cdot\nabla v}{2\int_{\partial B}\mu_{\widetilde{A}_r}|v|^2}\lesssim \frac{\int_B A_r\nabla u_r\cdot\nabla u_r}{\int_{\partial B}\mu_{A_r}|u_r|^2}\\ \lesssim N_{u_r}(B)=N_{u}(0,r)\lesssim N,
        \end{align*}
        which settles \textit{(i)}. In the last inequality we used almost monotonicity of the frequency function (\cite{GL1},\cite{GL2}). Part \textit{(ii)} follows immediately by the maximum principle and Claim \ref{maxest}.
    \end{proof}
\end{claim}

 We are now ready to finish the proof. Recall that $L=N\log^{2+\delta}N$. By Claim \ref{vestimates}, $v$ satisfies the hypothesis of Lemma \ref{analytic} with $\kappa=d/2$, so that $\sup_{(1-1/L)(1-1/N)B}|v|\geq  c_0\sup_{(1-1/N)B}|v|$. Note also that, since $N_v(B)\lesssim N$,
 \begin{align*}
     \sup_{(1-1/N)B}|v|\geq \|v\|_{L^2(\partial((1-1/N)B))}\gtrsim \|v\|_{L^2(\partial B)}.
 \end{align*}
  Using \eqref{aux} and choosing $a=a(\Lambda,\gamma,\delta)$ so that $\varepsilon=\varepsilon(a)$ is small enough, we can then estimate
\begin{align*}
    \sup_{(1-1/L)(1-1/N)B}|u_r|\geq \sup_{(1-1/L)(1-1/N)B}|v|-\sup_{(1-1/L)(1-1/N)B}|h|\\
    \geq c_0\sup_{(1-1/N)B}|v|-\varepsilon\|v\|_{L^2(\partial B)}
    \geq c_0\sup_{(1-1/N)B}|v|-\frac{1}{2}c_0\sup_{(1-1/N)B}|v|\\
    \geq \frac{1}{2}c_0\left(\sup_{(1-1/N)B}|u_r|-\sup_{(1-1/N)B}|h|\right)\geq \frac{1}{2}c_0\left(\sup_{(1-1/N)B}|u_r|-\varepsilon \|u_r\|_{L^2(\partial B)}\right)\\
    \geq \frac{1}{2}c_0\left(\sup_{(1-1/N)B}|u_r|-C\varepsilon \|u_r\|_{L^2(\partial((1-1/N)B))}\right)\geq \frac{1}{4}c_0\sup_{(1-1/N)B}|u_r|.
\end{align*}
Finally, we have that 
\begin{align*}
    \sup_{B(0,(1-1/L)(1-1/N)r)}|u|=\sup_{(1-1/L)(1-1/N)B}|u_r|\geq \tilde{c_0}\sup_{(1-1/N)B}|u_r|\\
    =\tilde{c_0}\sup_{B(0,(1-1/N)r)}|u|,
\end{align*}
which is equivalent to \eqref{aux1}.

\section{The inequality for eigenfunctions}
\label{seigen}
We now have all the tools needed to prove Theorem \ref{main}. We first state a sharper version for balls that have radius smaller than the wavelength. Throughout this section, $(M,g)$ will be a compact smooth Riemannian manifold, and $\varphi_{\ld}$ an eigenfunction of the Laplace-Beltrami operator on $M$, $\Delta_g\varphi_{\ld}+\ld\varphi_{\ld}=0$. 

\begin{proposition}
\label{small}
    Let $L\vcentcolon=\sqrt{\ld}\log^{2+\delta}\ld$ with $\delta>0$, and let $r_0\colon=\frac{a}{\sqrt{\ld}}$, where $a$ is a small constant depending on $(M,g)$ and $\delta$. Then for any geodesic ball $B_g(x,r)$ with $r<r_0$ we have 
    \begin{align}
        \label{smallballs1}
        \sup_{B_g\left(x,\left(1+\frac{1}{L}\right)r\right)}|\varphi_{\ld}|\leq C(M,g,\delta)\sup_{B_g(x,r)}|\varphi_{\ld}|;\\
        \label{smallballs2}
        \sup_{B_g(x,r)}|\nabla\varphi_{\ld}|\leq C(M,g,\delta)\frac{L}{r}\sup_{B_g(x,r)}|\varphi_{\ld}|.
    \end{align}
\end{proposition}

\begin{remark}
The result says that in balls of radius smaller than the wavelength we have a sharp Bernstein inequality up to logarithms. Inequality \eqref{smallballs1} (with $\sqrt{\ld}$ instead of $L$) was conjectured in \cite{DF2}, where an $L^2$ version is proved. 
\end{remark}

We first show how to obtain Theorem \ref{main} from Proposition \ref{small}, and then prove Proposition \ref{small}.

\begin{proof}[Proof of Theorem \ref{main}]
Take a ball $B_g(x,r)$; if $r<r_0$, we have nothing to prove. If $r>r_0$, let $y\in \overline{B_g(x,r)}$ be such that $|\nabla\varphi_{\ld}(y)|=\sup_{B_g(x,r)}|\nabla \varphi_{\ld}|$, and consider a ball $b\subset B_g(x,r)$ of radius $r_0$ such that $y\in \overline{b}$. Then, by Proposition \ref{small} applied to $b$,
\begin{align*}
    \sup_{B_g(x,r)}|\nabla \varphi_{\ld}|=|\nabla\varphi_{\ld}(y)|\lesssim \frac{L}{r_0}\sup_{b}|\varphi_{\ld}|\lesssim \ld\log^{2+\delta}\ld\sup_{B_g(x,r)}|\varphi_{\ld}|,
\end{align*}
where the implied constants depend on $(M,g)$ and $\delta$ only, and the theorem is proved.
\end{proof}

\begin{proof}[Proof of Proposition \ref{small}]
Let us first note that \eqref{smallballs2} follows easily from \eqref{smallballs1}. In fact, let $y\in\overline{B_g(x,r)}$ be such that $|\nabla \varphi_{\ld}(y)|=\sup_{B_g(x,r)}|\nabla\varphi_{\ld}|$. An application of standard elliptic estimates followed by \eqref{smallballs1} gives 
\begin{align*}
    |\nabla \varphi_{\ld}(y)|\lesssim \frac{L}{r}\sup_{B_g(y,\frac{r}{L})}|\varphi_{\ld}|\lesssim \frac{L}{r}\sup_{B_g(x,(1+\frac{1}{L})r)}|\varphi_{\ld}|\lesssim \frac{L}{r}\sup_{B_g(x,r)}|\varphi_{\ld}|,
\end{align*}
which is \eqref{smallballs2}. 

We will now prove \eqref{smallballs1} as a consequence of \eqref{aux1}. We consider a system of local coordinates on $M$, and identify the metric $g$ with a matrix $(g_{ij})$. We denote by $g^{-1}$ the inverse of this matrix, and set $|g|\colon=\det(g_{ij})$. We recall that the Laplace-Beltrami operator has the expression
\begin{align*}
    \Delta_g=\frac{1}{\sqrt{|g|}}\diver(\sqrt{|g|}g^{-1}\nabla(\cdot)),
\end{align*}
where the divergence and gradient are in the Euclidean sense. For $\varphi_{\ld}$ an eigenfunction of $\Delta_g$, consider the function 
\begin{align*}
    u(x,t)=\varphi_{\ld}(x)e^{\sqrt{\ld}t}.
\end{align*}
Note that $u$ is harmonic in $\widetilde{M}\colon=M\times\R$ with respect to the metric $\widetilde{g}=g\otimes \mathrm{dt}$, which means that in local coordinates $u$ satisfies the elliptic equation 
\begin{align*}
    \diver(\sqrt{|\widetilde{g}|}\widetilde{g}^{-1}\nabla u)=\diver(A\nabla u)=0.
\end{align*}
By results of Donnelly and Fefferman (\cite{DF}), $N_u(y,s)\leq C\sqrt{\ld}$ for any $s$ smaller than the injectivity radius of $M$. Now, \eqref{aux1} is essentially \eqref{smallballs1} but in Euclidean balls rather than geodesic balls. Let $\widetilde{x}=(x,0)$; now we consider normal coordinates in $\widetilde{M}$ centered at $\widetilde{x}$, and note that in these coordinates $A(\widetilde{x})=Id$. Denoting Euclidean balls with the subscript $e$, we observe that 
\begin{align*}
    B_e(\widetilde{x},r\sqrt{1-c_1r^2})\subset B_{\widetilde{g}}(\widetilde{x},r)\subset B_e(\widetilde{x},r\sqrt{1+c_2r^2})
\end{align*}
if $r$ is small enough. Our choice of $r_0$ then guarantees that
\begin{align*}
    B_{\widetilde{g}}\left(\widetilde{x},\left(1+\frac{1}{L}\right)r\right)\subset B_e\left(\widetilde{x},\left(1+\frac{2}{L}\right)r\right)
\end{align*}
and 
\begin{align*}
    B_e\left(\widetilde{x},\left(1-\frac{1}{L}\right)r\right)\subset B_{\widetilde{g}}\left(\widetilde{x},r\right)
\end{align*}
An application of \eqref{aux1} then gives 
\begin{align*}
    \sup_{B_{\widetilde{g}}\left(\widetilde{x},\left(1+\frac{1}{L}\right)r\right)}|u|\leq \sup_{B_e\left(\widetilde{x},\left(1+\frac{2}{L}\right)r\right)}|u|\lesssim \sup_{B_e\left(\widetilde{x},\left(1-\frac{1}{L}\right)r\right)}|u|\lesssim \sup_{B_{\widetilde{g}}\left(\widetilde{x},r\right)}|u|.
\end{align*}
Finally, using once again that $r<a/\sqrt{\ld}$, we have that 
\begin{align*}
    \sup_{B_{g}\left(x,\left(1+\frac{1}{L}\right)r\right)}|\varphi_{\ld}|\leq \sup_{B_{\widetilde{g}}\left(\widetilde{x},\left(1+\frac{1}{L}\right)r\right)}|u|\lesssim \sup_{B_{\widetilde{g}}\left(\widetilde{x},r\right)}|u|\\
    \lesssim \sup_{B_{g}\left(x,r\right)}|\varphi_{\ld}|e^{\sqrt{\ld}r}\lesssim \sup_{B_{g}\left(x,r\right)}|\varphi_{\ld}|,
\end{align*}
and \eqref{smallballs1} is proved. 
\end{proof}

\section{A stronger inequality in dimension two}
\label{s2d}
We assume now that the dimension of $M$ is two. Following the methods of Dong (\cite{Do}), we will show that there is a better Bernstein inequality than in higher dimension. 
Let $\varphi_\ld$ be an eigenfunction of the Laplace-Beltrami operator on $(M,g)$, $\Delta_g\varphi_\ld+\ld\varphi_\ld=0$. The main result of \cite{Do} can be formulated as  
\begin{align}\label{eq:Dongmain}
    \sup_{B_g(x,r)}|\nabla\varphi_{\ld}|\leq C(M,g)\max\left\{\frac{\sqrt{\ld}}{r},\ld^{\frac{3}{4}}\right\}\sup_{B_g(x,r)}|\varphi_{\ld}|.
\end{align}
It says that the sharp Bernstein inequality holds for scales up to $\ld^{-\frac{1}{4}}$.

We outline the proof given by Dong. Let
\begin{align*}
    q=|\nabla \varphi_{\ld}|^2+\frac{\ld}{2} |\varphi_{\ld}|^2.
\end{align*}
Denote by $K$ the Gaussian curvature of the surface $(M, g)$. It is shown in \cite{Do2} that
\begin{equation}\label{eq:Dongln}
\Delta_g \log q\ge -\lambda+2
\min\{K,0\}.
\end{equation}
Next, fix a point $x_0\in M$ and in a neighborhood of $x_0$ define the function
\[M(x)=\max\{q(y): d_g(y,x_0)\le d_g(x,x_0)\}.\]
Using \eqref{eq:Dongln} and standard properties of the Laplacian, it is not difficult to see that $\Delta_g M\ge -c\lambda$, see 
\cite{Do} for details.
The function $M$ depends only on the distance from $x$ to $x_0$ and is non-decreasing. Let $H$ be an upper bound for the absolute value of the sectional curvature of $M$.   Then the comparison theorem for the Laplace-Beltrami operators, see for example \cite[Lemma 7.1.9]{P} , implies
\begin{align*}
-c\lambda\le\Delta_g \log M\leq \Delta_H \log M,
\end{align*}
where $\Delta_H$ is the Laplace operator on the surface of constant sectional curvature $-H$. The metric on this surface is given by
 \[\mathrm{d}s^2_H=\mathrm{d}r^2+\rho_0^2(r)\mathrm{d}\theta^2\] where $\rho_0(r)=\frac{\sinh (\sqrt{H}r)}{\sqrt{H}}$. Denoting $\log M(x)=F(r)$ for $r=d_g(x,x_0)$, we conclude
\begin{align*}
\frac{1}{\rho_0}\frac{\mathrm{d}}{\mathrm{d}r}\left(\rho_0\frac{\mathrm{d}}{\mathrm{d}r}F(r)\right)\ge -c\ld.
\end{align*}
  Defining $t(r)=\int^{r}\frac{\mathrm{d}\tau}{\rho_0(\tau)}$, the inequality becomes:
\begin{align}
\label{in:5}
\frac{\mathrm{d}^2}{\mathrm{d}t^2}F\ge -c\lambda \rho_0^2.
\end{align}
The convexity estimate above can be combined with another result of \cite{Do2}, which says that
\begin{equation}\label{eq:Dongdouble}
 F(2r)-F(r)\lesssim \sqrt{\ld},   
\end{equation}
to derive the following inequality
\[
M(r)-M((1-\varepsilon)r)\lesssim (\sqrt{\ld}+\ld r^2)\varepsilon,\]
for small enough $r$ and any $\varepsilon>0$. 
For $r\lesssim \ld^{-1/4}$ the first term on the right hand side of the last inequality dominates and \eqref{eq:Dongmain} follows. 
Our improvement of \eqref{eq:Dongmain} is based on replacing the inequality \eqref{eq:Dongdouble} by a refined version of it for the case $r\gtrsim \lambda^{-1/2}.$

The main result of this section is the following
\begin{theorem}
\label{2dmain}
Let $(M,g)$ be a compact smooth Riemannian surface, and let $\varphi_{\ld}$ be a solution of $\Delta_g \varphi_{\ld} + \ld \varphi_{\ld}=0$. Let $B_g(x,r)\subset M$ any geodesic ball. Then we have:
\begin{align}
\label{2dbernstein}
    \sup_{B_g(x,r)}|\nabla\varphi_{\ld}|\leq C(M,g)\max\left\{\frac{\sqrt{\ld}}{r},\sqrt{\ld}\log\ld\right\}\sup_{B_g(x,r)}|\varphi_{\ld}|.
\end{align}
\end{theorem}

\begin{remark}
 Here we are able to refine Dong's method to get a sharp inequality for balls of radii up to $(\log\ld)^{-1}$. We do not know whether the $\log\ld$ is needed or it is a feature of the proof. 
\end{remark}

\begin{proof}
We fix a geodesic ball $B$ and denote by $x_0$ its center and by $r_0$ its radius. Our aim is to prove the inequality \eqref{2dbernstein} for this ball. We may assume that $r_0\gtrsim\ld^{-1/2}$, otherwise the inequality follows from \eqref{eq:Dongmain} or from our Theorem \ref{main}.

As above we set $M(x)=\sup_{y\in B_r}q(y)$, where $B_r$ is a geodesic ball of radius $r=d_g(x,x_0)$ centered at $x_0$. 
 Finally, we set $F(r)=\log M(r)$. \\

Our first aim is to estimate $F\left(r+\frac{1}{\sqrt{\lambda}}\right)-F(r)$. Let $a=1+\frac{1}{r\sqrt{\ld}}=1+\varepsilon$ so that $ar=r+1/\sqrt{\ld}$.  On the wave scale $\frac{1}{\sqrt{\ld}}$ the eigenfunction  $\varphi_{\ld}$ is approximately harmonic and the following elliptic estimates hold:
\begin{align}
\label{ell1}
\sup_{B_{ar}}|\varphi_{\ld}|^2\lesssim\ld\int_{B_{ar+1/\sqrt{\ld}}}|\varphi_{\ld}|^2; \\
\label{ell2}
\sup_{B_{ar}}|\nabla \varphi_{\ld}|^2\lesssim\ld^2\int_{B_{ar+1/\sqrt{\ld}}}|\varphi_{\ld}|^2.
\end{align}
We can then estimate:
\begin{align*}
\max_{B_{ar}}\left(|\nabla \varphi_{\ld}|^2+\frac{\lambda}{2} |\varphi_{\ld}|^2 \right)\lesssim\ld^2\int_{B_{ar+1/\sqrt{\ld}}}|\varphi_{\ld}|^2 +\ld^2\int_{B_{ar+1/\sqrt{\ld}}}|\varphi_{\ld}|^2 \\
\lesssim\ld^2\int_{B_{ar+1/\sqrt{\ld}}}|\varphi_{\ld}|^2.
\end{align*}
Iterating the estimate \eqref{eq:DFL2} of Donnelly and Fefferman, one gets
\begin{align*}
	\int_{B_{ar+1/\sqrt{\ld}}}|\varphi_{\ld}|^2\lesssim \left(a+\frac{1}{r\sqrt{\ld}}\right)^{C\sqrt{\lambda}}\int_{B_{r}}|\varphi_{\ld}|^2
	\lesssim\left(a+\frac{1}{r\sqrt{\lambda}}\right)^{C\sqrt{\lambda}}r^2M(x),
\end{align*} 
where $d_g(x,x_0)=r$.
Combining the last two inequalities and taking the logarithms, we get:
\begin{align*}
F(ar)-F(r)\lesssim \sqrt{\lambda}\log \left(a+\frac{1}{r\sqrt{\ld}}\right)+\log (\lambda r).
\end{align*}
Now we ask that $1/ \sqrt{\ld}\lesssim r\lesssim 1/\log\ld$. Then
\begin{align}
\label{in:4}
F\left(r+\frac{1}{\sqrt{\lambda}}\right)-F(r)= F(ar)-F(r)\lesssim\frac{1}{r}+\log\ld\lesssim \frac{1}{r}.
\end{align}
Note that the estimates up to now hold in any dimension. We will combine \eqref{in:4} with a two-dimensional inequality \eqref{in:5} of Dong.

Recall that we defined $t(r)$ by $t(r)=\int^r\frac{d\tau}{\rho_o(\tau)}$, where $\rho_0(\tau)=\frac{\sinh (\sqrt{H}\tau)}{\sqrt{H}}$ is a positive increasing function of $\tau$ and for small $\tau$ we have $\rho_0(\tau)\sim \tau$.  Let $t_1=t\left(r\left(1-\frac{1}{\sqrt{\lambda}}\right)\right)$, $t_2=t(r)$, and $t_3=t\left(r+\frac{1}{\sqrt{\lambda}}\right)$. The standard mean value theorem and \eqref{in:5} then imply 
\begin{align*}
\frac{ F\left(r+\frac{1}{\sqrt{\lambda}}\right)-F(r)}{t_3-t_2}-\frac{F(r)-F\left(r\left(1-\frac{1}{\sqrt{\lambda}}\right)\right)}{t_2-t_1} \gtrsim -\lambda \rho_0^2\left(r+\frac{1}{\sqrt{\lambda}}\right)(t_3-t_1).
\end{align*}
Now, 
\begin{align*}
t_2-t_1\sim\int_{r\left(1-\frac{1}{\sqrt{\lambda}}\right)}^r\frac{\mathrm{d}\tau}{\tau}\sim\frac{1}{\sqrt{\lambda}}; \\
t_3-t_2\sim\int_r^{r+\frac{1}{\sqrt{\lambda}}}\frac{\mathrm{d}\tau}{\tau}\sim\frac{1}{\sqrt{\lambda}r}\sim t_3-t_1.
\end{align*}
Using \eqref{in:4} in the previous inequality we thus get:
\begin{align*}
\frac{F(r)-F\left(r\left(1-\frac{1}{\sqrt{\lambda}}\right)\right)}{t_2-t_1}\lesssim \sqrt{\lambda}r\frac{1}{r}+\lambda r^2\frac{1}{\sqrt{\lambda}r}
\end{align*}
so that 
\begin{align*}
F(r)-F\left(r\left(1-\frac{1}{\sqrt{\lambda}}\right)\right)\lesssim 1.
\end{align*}

The rest of the proof is similar to one in \cite{Do}.
Set $r=r_0$ and $r_1\defeq r_0\left(1-\frac{1}{\sqrt{\lambda}}\right)$; we obtained:
\begin{align}
\label{in:6}
\tilde{M}(r)=M(x)\lesssim M(x_1)=\tilde{M}(r_1),
\end{align}
where $d_g(x,x_0)=r_0$ and $d_g(x_1,x_0)=r_1$.
Remember that the above holds for $\lambda^{-\frac{1}{2}}\lesssim r \lesssim (\log\lambda)^{-1}$. For any such $r$, consider $q(x^*)$ for some $x^*\in \overline{B_{r_1}}$. Then for $\sigma\sim \frac{r}{\sqrt{\lambda}}$ we have $ B_{\sigma}(x^*)\subset B_r$, and by elliptic estimates
\begin{align*}
|\nabla \varphi_{\ld}(x^*)|^2\lesssim \frac{1}{\sigma^4}\int_{ B_{\sigma}(x^*)}|\varphi_{\ld}|^2 \lesssim \frac{1}{\sigma^2}\max_{B_r}|\varphi_{\ld}|^2 \sim \frac{\lambda}{r^2}\max_{B_r}|\varphi_{\ld}|^2.
\end{align*}
This gives $\tilde{M}(r_1) \lesssim \lambda r^{-2}\sup_{B_r}|\varphi_{\ld}|^2$, and finally, by \eqref{in:6},
\begin{align}
\label{in:7}
\sup_{B(x_0,r_0)}|\nabla \varphi_{\ld}|\leq \sqrt{\tilde{M}(r_0)} \lesssim \sqrt{\tilde{M}(r_1)} \lesssim \frac{\sqrt{\lambda}}{r_0}\sup_{B(x_0,r_0)}|\varphi_{\ld}|,
\end{align}
which is the sharp Bernstein inequality for $\lambda^{-\frac{1}{2}}\lesssim r_0 \lesssim (\log\lambda)^{-1}$. 

For $r_0\gtrsim (\log\lambda)^{-1}$, we can take $r'\sim (\log\lambda)^{-1}$ such that $\sup_{B(x_0,r_0)}|\nabla \varphi_{\ld}|=\sup_{B(p,r')}|\nabla \varphi_{\ld}|$ and $B(p, r')\subset B(x_0, r_0)$. Then applying \eqref{in:7} to $B(p, r')$ we obtain:
\begin{align*}
\sup_{B(x_0,r_0)}|\nabla \varphi_{\ld}|=\sup_{B(p,r')}|\nabla \varphi_{\ld}| \lesssim \frac{\sqrt{\lambda}}{r'}\sup_{B(p,r')}|\varphi_{\ld}| \lesssim {\sqrt{\lambda}\log \lambda}\sup_{B(x_0, r_0)}|\varphi_{\ld}|,
\end{align*}
which concludes the proof of Theorem \ref{2dmain}. 
\end{proof}

\section*{Appendix}

The goal of this section is to prove a quite sharp $L^p$ version of the Bernstein inequality for harmonic functions in $\R^d$. We start with a Bernstein inequality for harmonic polynomials, which was proved in \cite{M}. The sharp inequality for the supremum norm  in dimensions two and three was studied by Szeg\"o \cite{S} several decades earlier.

	\begin{lemma}\label{poly} Let $P_N$ be a harmonic polynomial of degree $N$ in $\R^d$ and let $1\le p\le\infty$. Then
		\begin{equation} \|\nabla P_N\|_{L^p(B(0,r))}\lesssim \frac{N}{r}\|P_N\|_{L^p(B(0,r))}.
		\end{equation}
	
		\end{lemma}

\begin{proof}
The inequality follows from Lemma 4.2 in \cite{M}, where it is proved that
\[
\|P_N\|_{L^p(\partial B(0, (1+\frac{1}{N})r))}\lesssim \|P_N\|_{L^p(\partial B(0,r))}.
\]
This easily implies
 \begin{equation}\label{eq:M}
 \|P_N\|_{L^p(B(0, (1+\frac{1}{N})r))}\lesssim \|P_N\|_{L^p(B(0,r))}.\end{equation}
   The Cauchy estimate and properties of subharmonic functions imply
 \[|\nabla P_N(x)|\lesssim \frac{N}{r}\strokedint_{B(x, \frac{r}{N})}|P_N|.\]
 Then, applying H\"older's inequality, we get
 \[\int_{B(0,r)}|\nabla P_N|^p\lesssim \frac{N^p}{r^p}\int_{B(0, (1+\frac{1}{N})r)}|P_N|^p.\]
 	Finally, \eqref{eq:M} gives the required inequality.
\end{proof}


We will now extend the Bernstein inequality to harmonic functions. The role of the degree of a polynomial will be played
by the frequency function, which we define as follows:
\begin{equation}\label{eq:Freq} 
N_h(x,r)=\frac{r\int_{B(x,r)}|\nabla h|^2}{2\int_{\partial B(x,r)}|h|^2}
\end{equation}
	for a function $h\in C(\overline{B(0,r)})$ harmonic in $B(x,r)$. Note that for a homogeneous harmonic polynomial $P$ we have $N_P(0,r)=N$. In general, $N_h(x,r)$ is a non-decreasing function of $r$ and 
	\begin{equation}\label{eq:double}
\strokedint_{B(0,2r)}|h|^2\le 2^{2N_h(0,2r)}\strokedint_{B(0,r)}|h|^2 
	\end{equation}
	for any function $h$ harmonic in $B(0,R)$ with $R>2r$; see for instance \cite[Corollary 1.5]{H} for a proof. We also note that the mean value property for harmonic functions implies that
	\begin{equation}\label{eq:ell}
\left(	\strokedint_{B(0,r)}|h|^2\right)^{1/2}\le \sup_{B(0,r)}|h|\lesssim \strokedint_{B(0,2r)}|h|.
	\end{equation}

	The main result of this section is the following Bernstein inequality for harmonic functions:
	\begin{theorem}\label{thmx}
		Let $h$ be a harmonic function in $B(0,R)\subset\R^d$ such that $N_h(0,\rho)\le N$ for any  $\rho<R$  and let $r<R/2$. Then  
		\begin{equation}
		\|\nabla h\|_{L^p(B(0,r))}\lesssim \frac{N}{r}\|h\|_{L^p(B(0,r))},
		\end{equation}
		for any $p\in[1,\infty]$.
	\end{theorem}
	
	The derivation of Theorem \ref{thmx} from Lemma \ref{poly} can also be found in \cite{HL94}. To prove the theorem we approximate a harmonic function with frequency bounded by $N$ by a harmonic polynomial of degree $5N$. 
	
	\begin{lemma} \label{l:pol}
	Suppose that $a>0$ is a fixed constant.  There exists $N_0=N_0(a,d)$ such that for any  function $h\in C(\overline{B(0,2r)})$ harmonic   in $B(0,2r)$ which satisfies $N_h(0,2r)\le N$ for  some $N>N_0$, there is a harmonic polynomial $P$ of degree at most $5N$ such that
	\[\sup_{B(0,(1+\frac{1}{N})r)}|h-P|\le a\strokedint_{B(0,r)}|h|.\]
	\end{lemma}
	
	\begin{proof} For $x\in B(0,2r)$ we write $x=\xi\rho$, where $\rho=|x|$.
	We can decompose $h$ into spherical harmonics,
	\[h(x)=\sum_{k=0}^\infty a_kY_k(\xi)\rho^k,\]
	where $\strokedint_{\partial B(0,1)}Y_jY_k=\delta_{jk}$.
	Note that
	$\strokedint _{\partial B(0,\rho)} h^2=\sum_k a_k^2\rho^{2k}$ and

	\[N_h(0,2r)=\frac{\sum_{k=1}^\infty ka_k^2(2r)^{2k}}{\sum_{k=0}^\infty a_k^2(2r)^{2k}}.\]
We also have
\begin{equation}\label{eq:Nball}
\strokedint_{B(0,r)}|h|^2=\sum_{k=0}^\infty\frac{d}{2k+d}|a_k|^2r^{2k}.
\end{equation}
Consider the decomposition
\[h(x)=P(x)+Q(x)=\sum_{k=0}^{5N} a_kY_k(\xi)\rho^k+\sum_{k=5N+1}^\infty a_kY_k(\xi)\rho^k.\]
The frequency bound $N_h(0,2r)\le N$ implies that
\[\sum_k ka_k^2(2r)^{2k}\le N\sum_k a_k^2(2r)^{2k}.\]
Therefore we have:
\[5N\sum_{k=5N+1}^\infty a_k^2(2r)^{2k}\le \sum_{k=5N+1}^\infty ka_k^2(2r)^{2k}\le N\sum_{k=0}^{5N} a_k^2(2r)^{2k}+N\sum_{k=5N+1}^\infty a_k^2(2r)^{2k}
.\]	
The above inequality and \eqref{eq:Nball}  give
	\begin{multline*}
	\sum_{k=5N+1}^\infty a_k^2(2r)^{2k}\le \frac{1}{4}\sum_{k=0}^{5N} a_k^2(2r)^{2k}\le \frac{10N+d}{4d}\sum_{k=0}^\infty \frac{d}{2k+d}a_k^2(2r)^{2k}\\\le \frac{10N+d}{4d}
	\strokedint_{B(0,2r)}|h|^2.
\end{multline*}
	Now, applying \eqref{eq:double} twice and then \eqref{eq:ell} for $B(0,r/2)$, we get
	\begin{equation}\label{eq:MainL}
	\sum_{k=5N+1}^\infty a_k^2(2r)^{2k}\le C 2^{4N}N\left(\strokedint_{B(0,r)}|h|\right)^2,
	\end{equation}
	when $N>N_0$.

Our aim is to estimate  the norm of $Q=h-P$.	 Let us first remind the reader that we have the following pointwise bounds for ($L^2$
normalized) spherical harmonics: 
\[|Y_k(\xi)|\le Ck^{\frac{d-2}{2}}.\]
This would follow from general results of H{\"o}rmander on $L^{\infty}$ bounds for eigenfunctions, but since for spherical harmonics the proof is much simpler we briefly outline the argument. Let $\mathcal{H}_k$ denote the space of spherical harmonics of degree $k$, i.e., corresponding to eigenvalue $k(k+d-2)$, and let $Z_k(\xi,\eta)$ denote the reproducing kernel for $\mathcal{H}_k$, also called the zonal spherical harmonic. Then, denoting by angular brackets the $L^2$ pairing on $S^{d-1}$, we have 
\begin{align*}
    |Y_k(\xi)|=|\langle Y_k(\eta), Z_k(\xi,\eta)\rangle|\leq \|Y_k\|_2\|Z_k(\xi,\cdot)\|_2\leq \sqrt{dim(\mathcal{H}_k)}\leq Ck^{\frac{d-2}{2}},
\end{align*}
where for the second and third inequalities we refer to \cite[Chapter 5]{ABR}.

	Let $b=1+\frac{1}{N}$. From the estimates on spherical harmonics it follows that, for  $x\in B(0,br)$,
	\[|Q(x)|\lesssim \sum_{k=5N+1}^\infty |a_k|k^{\frac{d-2}{2}}(br)^k.\]
	Applying the Cauchy inequality and \eqref{eq:MainL}, we get
\begin{multline*}
|Q(x)|\lesssim \left(\sum_{5N+1}^\infty a_k^2(2r)^{2k}\right)^{1/2}\left(\sum_{5N+1}^\infty k^{d-2}\left(\frac{b}{2}\right)^{2k}\right)^{1/2}\\ \lesssim N^{\frac12} 2^{2N}\left(\sum_{5N+1}^\infty k^{d-2}\left(\frac{b}{2}\right)^{2k}\right)^{1/2}\strokedint_{ B(0,r)}|h|.\end{multline*}
Comparing the series $\sum k^{d-2} (b/2)^{2k}$ to a geometric series, we conclude that, for $N>N_0$,
\[\sum_{k=5N+1}^\infty k^{d-2}\left(\frac{b}{2}\right)^{2k}\le (5N)^{d-2} 2^{-5N}.\] This implies 
\[|Q(x)|\lesssim N^{\frac{d-1}{2}}2^{-\frac{N}{2}}\strokedint_{B(0,r)}|h|.\]
We complete the proof of the lemma by choosing $N_0=N_0(a,d)$ large enough.
\end{proof}

\begin{remark}
The proof implies the following statement. There is $q<1$ such that if $u$ is a harmonic function and $N_u(0,2r)\le N$, $N\ge 2$, then there is a harmonic polynomial $P$ of degree not exceeding $5N$ such that
\[\sup_{B(0,(1+\frac{1}{N})r)}|h-P|\le Cq^N\strokedint_{B(0,r)}|h|.\]
The celebrated Bernstein-Walsh theorem implies that harmonic functions in a domain $\Omega$ can be approximated  uniformly by polynomials on a compact subset of $\Omega$ and the error of approximation decay exponentially, see for example \cite{BL}. Lemma \ref{l:pol} can be also deduced from this result and inequalities \eqref{eq:double}  and \eqref{eq:ell}.
\end{remark}
\begin{proof}[Proof of Theorem \ref{thmx}]
As in the proof of Lemma \ref{poly} it suffices to show that 
	\[\|h\|_{L^p(B(0, (1+\frac{1}{N})r))}\le C\|h\|_{L^p(B(0,r))}.\]
	
Let $b=1+\frac{1}{N}$. We apply the above lemma with $a=b^{-d}/2$. It gives $h=P+Q$, where $P$ is a harmonic polynomial of degree $d\le\max\{5N,5N_0\}$ and $Q$ satisfies
\[\sup_{B(0,br)}|Q|\le \frac{1}{2b^d}\strokedint_{B(0,r)}|h|.\]
Then, for any $p\in[1,\infty]$, by H\"older's inequality, we have
 \[\|Q\|_{L^p(B(0,br))}\le \frac{1}{2}\|h\|_{L^p(B(0,r))}\le \frac{1}{2}\|h\|_{L^p(B(0,br))}.\]
Finally,  the triangle inequality and \eqref{eq:M} imply the required estimate:
\begin{multline*}\|h\|_{L^p(B(0,br))}\le 2\|P\|_{L^p(B(0,br))}\le CN\|P\|_{L^p(B(0,r))}
\le 2CN\|h\|_{L^p(B(0,r))}.
\end{multline*}

\end{proof}

\section*{Acknowledgements}
Part of the research for this article was conducted while the first named author was a Visiting Student Researcher at the Department of Mathematics of Stanford University; it is a pleasure to thank the department for the hospitality. \\
This work was concluded when both authors were guests of the Hausdorff Research Institute for Mathematics in Bonn during the trimester program "Interactions between Geometric measure theory, Singular integrals, and PDE"; we thank the Institute and the organizers of the trimester program for providing excellent working conditions and a stimulating environment. \\
We also wish to thank Joaquim Ortega-Cerd{\`a} for pointing out the reference \cite{M} to us, and Jiuyi Zhu for the reference \cite{HL94}.\\
Both authors are supported by Project 275113 of the Research Council of Norway. The second named author is supported by NSF grant DMS-1956294.

\bibliographystyle{plain}
\bibliography{references}
\end{document}